\documentclass[11pt]{article}
\usepackage{amsfonts,amscd,amssymb, amsmath}

\newtheorem{definition}{Definition}[section]

\newtheorem{proposition}[definition]{Proposition}
\newtheorem{corollary}[definition]{Corollary}
\newtheorem{remark}[definition]{Remark}

\newtheorem{theorem}[definition]{Theorem}

\newcommand{\mfa}{\mbox{$\mf {a}$}}

\def\rawo\lonra{\longrightarrow}

\def\ot{\otimes}
\def\hot{\hat{\otimes }}
\def\tot{\tilde{\otimes }}
\def\mfa{\mathfrak{a}}

\newcommand{\selabel}[1]{\label{se:#1}}

\allowdisplaybreaks[4]

\newenvironment{proof}{{\it Proof.}}{\hfill $ \square $ \vskip 4mm}

\begin{document}
\title{Yetter-Drinfeld modules for Hom-bialgebras}
\author{Abdenacer Makhlouf\\
Universit\'{e} de Haute Alsace, \\
Laboratoire de Math\'{e}matiques, Informatique et Applications, \\
4, rue des fr\`{e}res Lumi\`{e}re, F-68093 Mulhouse, France\\
e-mail: Abdenacer.Makhlouf@uha.fr
\and Florin Panaite\thanks {Work supported by a grant of the Romanian National 
Authority for Scientific Research, CNCS-UEFISCDI, 
project number PN-II-ID-PCE-2011-3-0635,  
contract nr. 253/5.10.2011.}\\
Institute of Mathematics of the
Romanian Academy\\
PO-Box 1-764, RO-014700 Bucharest, Romania\\
 e-mail: Florin.Panaite@imar.ro}
\date{}
\maketitle

\begin{abstract}
The aim of this paper is to define and study Yetter-Drinfeld modules over Hom-bialgebras, a 
generalized version of bialgebras obtained by modifying  the algebra and coalgebra structures  by a 
homomorphism. Yetter-Drinfeld modules over a Hom-bialgebra with bijective structure map provide 
solutions of the Hom-Yang-Baxter equation. The category $_H^H{\mathcal YD}$ of Yetter-Drinfeld modules 
with bijective structure maps over a Hom-bialgebra $H$ with bijective structure map can be organized, in two 
different ways, as a quasi-braided pre-tensor category. If $H$ is quasitriangular (respectively coquasitriangular) 
the first (respectively second) quasi-braided pre-tensor category $_H^H{\mathcal YD}$ contains, as a 
quasi-braided pre-tensor subcategory, the category of modules (respectively comodules) with bijective 
structure maps over $H$.
\end{abstract}
\section*{Introduction}
${\;\;\;}$The first examples of Hom-type algebras were related to $q$-deformations of Witt and Virasoro 
algebras,
 which play an important r\^{o}le in Physics, mainly in conformal field theory. In a theory with 
conformal symmetry, 
the Witt algebra $W$ is a part of the complexified Lie algebra 
$Vect^\mathbb{C}(\mathbb{S}) \times Vect^\mathbb{C}(\mathbb{S})$, where $\mathbb{S}$ is the 
unit circle, belonging to the classical conformal symmetry. 
The central extensions of $W$ by $\mathbb{C }$ become important for the quantization process. 
The $q$-deformations of Witt and Virasoro algebras are obtained when  the  derivation is replaced by a 
$\sigma$-derivation. It was observed in the pioneering works   
\cite{AizawaSaito,ChaiElinPop,ChaiKuLukPopPresn,ChaiIsKuLuk,ChaiPopPres,
CurtrZachos1,DaskaloyannisGendefVir,Hu, Kassel1,LiuKeQin} that they are no longer  Lie algebras.  
Motivated by these
examples and their generalization, Hartwig, Larsson and Silvestrov in \cite{HLS,LS1,LS2,LS3} 
introduced the notion of  Hom-Lie algebra as a deformation of Lie algebras in which the Jacobi identity 
is twisted by a homomorphism. The  associative-type objects corresponding to Hom-Lie algebras, called 
 Hom-associative algebras, have been introduced in \cite{ms1}. Usual functors between the categories of 
Lie algebras and associative algebras have been extended to the Hom-setting. It was shown in  \cite{ms1} 
that a commutator of a Hom-associative algebra gives rise to a Hom-Lie algebra; the construction of the free 
Hom-associative algebra and the  enveloping algebra of a Hom-Lie algebra have been provided in 
\cite{Yau:EnvLieAlg}. Since then, Hom-analogues of various classical structures and results  have 
been introduced and discussed by many authors. For instance, representation theory,   cohomology and 
deformation theory  for Hom-associative algebras and Hom-Lie algebras have been developed in 
\cite{AEM,ms2, Sheng}. See also \cite{fgs, Gohr} for other properties of Hom-associative algebras.
All these generalizations coincide with the usual definitions when the structure map equals the identity.

 The dual concept of Hom-associative algebras, called Hom-coassociative coalgebras, as well as 
Hom-bialgebras and Hom-Hopf algebras, have been introduced in \cite{ms3,ms4} and also studied  in 
\cite{stef,yau2}.  As expected, the enveloping Hom-associative algebra of a Hom-Lie algebra is 
naturally a Hom-bialgebra. A twisted version of module algebras called module Hom-algebras has 
been studied in \cite{yau1}, 
where $q$-deformations of the $\mathfrak{sl}(2)$-action on the affine plane were provided. 
Objects admitting coactions by Hom-bialgebras have been studied first in \cite{yau2}. A matrix Hom-associative 
algebra was endowed with a Hom-bialgebra structure $H$ and examples of $H$-comodule Hom-algebra 
structures on the Hom-affine plane $\textbf{A}_2$ have been provided.   
 In \cite{yauhomyb1,yauhomyb2,yauhomyb3}, various generalizations of Yang-Baxter equations and 
related algebraic structures have been studied. D. Yau provided solutions of HYBE, a twisted 
version of the Yang-Baxter equation called the Hom-Yang-Baxter equation, from Hom-Lie algebras, 
quantum enveloping algebra of $\mathfrak{sl}(2)$, the Jones-Conway polynomial, 
Drinfeld's (co)quasitriangular bialgebras and Yetter-Drinfeld modules (over bialgebras). 
It was shown that solutions of 
HYBE can be extended to operators that satisfy the braid relations, which can then be 
used to construct representations of the braid group, in case an invertibility condition holds. 
Moreover, a generalization of the classical Yang-Baxter equation and its connection  to Hom-Lie bialgebras 
have been explored. See also \cite{BEM} for other results related to Hom-Lie bialgebras.   
In the series of papers \cite{homquantum1,homquantum2,homquantum3}, D. Yau  studied  Hom-type 
generalizations of (co)quasitriangular bialgebras, quantum groups and the quantum 
Yang-Baxter equation (QYBE). It was shown that quasitriangular and coquasitriangular Hom-bialgebras 
come equipped with a solution of the quantum Hom-Yang-Baxter equation (QHYBE) or 
operator quantum Hom-Yang-Baxter equation (OQHYBE). Examples of quasitriangular Hom-bialgebras 
have been 
given, including  Drinfeld's quantum enveloping algebra $\mathcal{U}_h(\mathfrak{g})$  of a semi-simple Lie 
algebra or a Kac-Moody algebra $\mathfrak{g}$ and anyonic quantum groups.
In \cite{Elhamdadi-Makhlouf}, Hom-quasi-bialgebras have been introduced and concepts like gauge 
transformation and Drinfeld twist  generalized. 
Moreover,  an example of a twisted quantum double was provided.  One of the main tool to construct 
examples is the ''twisting principle'' introduced by D. Yau for Hom-associative algebras and since then 
extended to various Hom-type algebras. It allows to construct a Hom-type algebra starting from a 
classical-type algebra and an algebra homomorphism.

The aim of this paper is to introduce and study Yetter-Drinfeld modules over a Hom-bialgebra $H$, as objects 
$(M, \alpha _M)$ such that $(M, \alpha _M)$ is both a left $H$-module and a left $H$-comodule and a 
certain compatibility condition between the two structures holds. This condition was chosen in 
such a way that if $(M, \alpha _M)$ is a left module over a quasitriangular Hom-bialgebra or a left comodule over a 
coquasitriangular Hom-bialgebra then $M$ becomes a Yetter-Drinfeld module over that Hom-bialgbra.  
We will denote by $_H^H{\mathcal YD}$ the category whose objects are Yetter-Drinfeld modules 
$(M, \alpha _M)$ with $\alpha _M$ bijective over a Hom-bialgebra $H$ with bijective structure map $\alpha _H$. 

The paper is organized as follows. In Section 1, we review the main definitions and properties of pre-tensor 
categories, Hom-associative algebras, Hom-bialgebras and related structures. In Section 2, we introduce 
Yetter-Drinfeld modules and discuss some elementary aspects. We extend the twisting principle to 
Yetter-Drinfeld modules over bialgebras and we show that Yetter-Drinfeld modules over a Hom-bialgebra 
with bijective structure map give rise to solutions of the HYBE. In Section 3, we prove that 
$_H^H{\mathcal YD}$ can be organized as a quasi-braided pre-tensor category (with nontrivial associators), 
for which the quasi-braiding satisfies the usual braid relation (besides the dodecagonal braid relation involving 
the associators). It turns out that, if $H$ is a quasitriangular Hom-bialgebra, the category of left $H$-modules 
with bijective structure maps is a quasi-braided pre-tensor subcategory of $_H^H{\mathcal YD}$. 
In Section 4, we find another quasi-braided pre-tensor category structure on $_H^H{\mathcal YD}$, 
with the property that if $H$ is a coquasitriangular Hom-bialgebra then $_H^H{\mathcal YD}$ contains the 
category of left $H$-comodules with bijective structure maps as a quasi-braided pre-tensor category. 
\section{Preliminaries}\selabel{1}
${\;\;\;}$
We work over a base field $k$. All algebras, linear spaces
etc. will be over $k$; unadorned $\ot $ means $\ot_k$. For a comultiplication 
$\Delta :C\rightarrow C\ot C$ on a vector space $C$ we use a 
Sweedler-type notation $\Delta (c)=c_1\ot c_2$, for $c\in C$. Unless 
otherwise specified, the (co)algebras ((co)associative or not) that will appear 
in what follows are {\em not} supposed to be (co)unital, and a multiplication 
$\mu :V\ot V\rightarrow V$ on a linear space $V$ is denoted by juxtaposition: 
$\mu (v\ot v')=vv'$. 

We recall now several concepts and results, fixing thus the terminology 
to be used in the rest of the paper. 
\begin{definition} (\cite{li}) A {\em pre-tensor category} is a category satisfying all the axioms 
of a tensor category in \cite{kassel} except for the fact that we do not require the existence of a 
unit object and of left and right unit constraints. If $({\mathcal C}, \otimes, a)$ is a pre-tensor 
category, a {\em quasi-braiding} $c$ in ${\mathcal C}$ is a family of natural morphisms 
$c_{V, W}:V\ot W\rightarrow W\ot V$ in ${\mathcal C}$ satisfying all the axioms of 
a braiding in \cite{kassel} except for the fact that we do not require $c_{V, W}$ to be isomorphisms; 
in this case, $({\mathcal C}, \otimes, a, c)$ is called a quasi-braided pre-tensor category. 
\end{definition}

Exactly as for usual braided categories, a quasi-braiding on a pre-tensor category satisfies the dodecagonal 
braid relation in \cite{kassel}, p. 317.  
\begin{definition} \label{defYD}
Let $H$ be a bialgebra and $M$ a linear space 
which is a left $H$-module with 
action $H\ot M\rightarrow M$, $h\ot m\mapsto h\cdot m$ and a 
left $H$-comodule with coaction $M\rightarrow H\ot M$, $m\mapsto m_{(-1)}\ot m_{(0)}$. 
Then $M$ is called a (left-left) Yetter-Drinfeld module over $H$  if the 
following compatibility condition holds, for all $h\in H$, $m\in M$:
\begin{eqnarray}
&&(h_1\cdot m)_{(-1)}h_2\ot (h_1\cdot m)_{(0)}=
h_1m_{(-1)}\ot h_2\cdot m_{(0)}. 
\label{asocYD}
\end{eqnarray}
\end{definition}

We summarize several definitions and properties about Hom-type structures. Since various authors use different terminology, 
some caution is necessary. In what follows, we use terminology as in Yau's paper \cite{yau1}, which is different 
from the original terminology in \cite{ms1}, \cite{ms2} (where no extra assumption on the linear map 
$\alpha $ is made) and also different from Yau's paper \cite{yau2}, where for instance the multiplicativity of the 
map $\alpha $ is emphasized by calling ''multiplicative Hom-associative algebra'' what we will 
call for simplicity  in what follows ''Hom-associative algebra''. 
\begin{definition}
(i) A {\em Hom-associative algebra} is a triple $(A, \mu , \alpha )$, in which $A$ is a linear space, 
$\alpha :A\rightarrow A$  and $\mu :A\ot A\rightarrow A$ are linear maps,  
with notation $\mu (a\ot a')=aa'$, satisfying the following conditions, for all $a, a', a''\in A$:
\begin{eqnarray*}
&&\alpha (aa')=\alpha (a)\alpha (a'), \;\;\;\;\;(multiplicativity)\\
&&\alpha (a)(a'a'')=(aa')\alpha (a''). \;\;\;\;\;(Hom-associativity)
\end{eqnarray*}
We call $\alpha $ the {\em structure map} of $A$. 

A morphism $f:(A, \mu _A , \alpha _A)\rightarrow (B, \mu _B , \alpha _B)$ of Hom-associative algebras 
is a linear map $f:A\rightarrow B$ such that $\alpha _B\circ f=f\circ \alpha _A$ and 
$f\circ \mu_A=\mu _B\circ (f\ot f)$. \\
(ii) A {\em Hom-coassociative coalgebra} is a triple $(C, \Delta, \alpha )$, in which $C$ is a linear 
space, $\alpha :C\rightarrow C$ and $\Delta :C\rightarrow C\ot C$ are linear maps, 
satisfying the following conditions:
\begin{eqnarray*}
&&(\alpha \ot \alpha )\circ \Delta =
\Delta \circ \alpha , \;\;\;\;\;(comultiplicativity)\\  
&&(\Delta \ot \alpha )\circ \Delta =
(\alpha \ot \Delta )\circ \Delta . \;\;\;\;\;(Hom-coassociativity)
\end{eqnarray*}

A morphism $g:(C, \Delta _C , \alpha _C)\rightarrow (D, \Delta _D , \alpha _D)$ of Hom-coassociative 
coalgebras  
is a linear map $g:C\rightarrow D$ such that $\alpha _D\circ g=g\circ \alpha _C$ and 
$(g\ot g)\circ \Delta _C=\Delta _D\circ g$.
\end{definition}
\begin{remark}
Assume that $(A, \mu _A , \alpha _A)$ and $(B, \mu _B, \alpha _B)$ are two Hom-associative algebras; then 
$(A\ot B, \mu _{A\ot B}, \alpha _A\ot \alpha _B)$ is a Hom-associative algebra (called the tensor 
product of $A$ and $B$), where $\mu _{A\ot B}$ is the usual multiplication: 
$(a\ot b)(a'\ot b')=aa'\ot bb'$. 
\end{remark}
\begin{definition} (\cite{ms3},\cite{yau1}, \cite{homquantum3}) 
(i) Let $(A, \mu _A , \alpha _A)$ be a Hom-associative algebra, $M$ a linear space and $\alpha _M:M
\rightarrow M$ a linear map. A {\em left $A$-module} structure on $(M, \alpha _M)$ consists of a linear map 
$A\ot M\rightarrow M$, $a\ot m\mapsto a\cdot m$, satisfying the conditions:
\begin{eqnarray}
&&\alpha _M(a\cdot m)=\alpha _A(a)\cdot \alpha _M(m), \label{hommod1}\\
&&\alpha _A(a)\cdot (a'\cdot m)=(aa')\cdot \alpha _M(m), \label{hommod2}
\end{eqnarray} 
for all $a, a'\in A$ and $m\in M$. If $(M, \alpha _M)$ and $(N, \alpha _N)$ are left $A$-modules (both 
$A$-actions denoted by $\cdot$),  
a morphism of left $A$-modules $f:M\rightarrow N$ is a linear map satisfying the conditions 
$\alpha _N\circ f=f\circ \alpha _M$  and $f(a\cdot m)=a\cdot f(m)$, for all $a\in A$ and $m\in M$. \\
(ii) Let $(C, \Delta _C , \alpha _C)$ be a Hom-coassociative coalgebra, $M$ a linear space and $\alpha _M:M
\rightarrow M$ a linear map. A {\em left $C$-comodule} structure on $(M, \alpha _M)$ consists of a linear map 
$\lambda :M\rightarrow C\ot M$ (usually denoted by $\lambda (m)=m_{(-1)}\ot m_{(0)}$)
satisfying the following conditions:
\begin{eqnarray}
&&(\alpha _C\ot \alpha _M)\circ \lambda =\lambda \circ \alpha _M,  
\label{leftcom1}\\
&&(\Delta _C\ot \alpha _M)\circ \lambda =(\alpha _C\ot \lambda )\circ \lambda . 
\label{leftcom2}
\end{eqnarray} 
If $(M, \alpha _M)$ and $(N, \alpha _N)$ are left $C$-comodules, with structures 
$\lambda _M:M\rightarrow C\ot M$ and $\lambda _N:N\rightarrow C\ot N$, 
a morphism of left $C$-comodules $g:M\rightarrow N$ is a linear map satisfying the conditions 
$\alpha _N\circ g=g\circ \alpha _M$  and $(id_C\ot g)\circ \lambda _M=\lambda _N\circ g$. 
\end{definition}

\begin{definition} (\cite{ms3}, \cite{ms4})
A {\em Hom-bialgebra} is a quadruple $(H, \mu , \Delta, \alpha )$, in which $(H, \mu , \alpha )$ is 
a Hom-associative algebra, $(H, \Delta , \alpha )$ is a Hom-coassociative coalgebra  
and moreover $\Delta $ is a morphism of Hom-associative algebras.  
\end{definition}

In other words, a Hom-bialgebra is a Hom-associative algebra $(H, \mu , \alpha )$ endowed with a 
linear map $\Delta :H\rightarrow H\ot H$, with notation $\Delta (h)=h_1\ot h_2$, such that the 
following conditions are satisfied, for all $h, h'\in H$: 
\begin{eqnarray}
&&\Delta (h_1)\ot \alpha (h_2)=\alpha (h_1)\ot \Delta (h_2),  \label{hombia1}\\
&&\Delta (hh')=h_1h'_1\ot h_2h'_2, \label{hombia2}\\
&&\Delta (\alpha (h))=\alpha (h_1)\ot \alpha (h_2). \label{hombia3}
\end{eqnarray}

The following result provides a way to construct examples of Hom-associative algebras, Hom-coassociative 
coalgebras or Hom-bialgebras. It is called the ''twisting principle'' or sometimes a composition method.
\begin{proposition} (\cite{ms4}, \cite{yau3})
(i) Let $(A, \mu )$ be an associative algebra and $\alpha :A\rightarrow A$ an algebra endomorphism. Define 
a new multiplication $\mu _{\alpha }:=\alpha \circ \mu :A\ot A\rightarrow A$. Then 
$(A, \mu _{\alpha }, \alpha )$ is a Hom-associative algebra, denoted by $A_{\alpha }$. \\
(ii) Let $(C, \Delta )$ be a coassociative coalgebra and $\alpha :C\rightarrow C$ a 
coalgebra endomorphism. Define 
a new comultiplication $\Delta _{\alpha }:=\Delta \circ \alpha :C\rightarrow C\ot C$. Then 
$(C, \Delta _{\alpha }, \alpha )$ is a Hom-coassociative coalgebra, denoted by $C_{\alpha }$. \\
(iii) Let $(H, \mu , \Delta )$ be a bialgebra and $\alpha :H\rightarrow H$ a bialgebra endomorphism. 
If we define $\mu _{\alpha }$ and $\Delta _{\alpha }$ as in (i) and (ii), then 
$H_{\alpha }=(H, \mu _{\alpha }, \Delta _{\alpha }, \alpha )$ is a Hom-bialgebra. 
\end{proposition}
\begin{proposition} (\cite{homquantum3}) \label{tensprodmod}
Let $(H, \mu _H, \Delta _H, \alpha _H)$ be a Hom-bialgebra. \\
(i) If $(M, \alpha _M)$ and 
$(N, \alpha _N)$ are left $H$-modules, then $(M\ot N, \alpha _M\ot \alpha _N)$ is also a 
left $H$-module, with $H$-action defined by $H\ot (M\ot N)\rightarrow M\ot N$, $h\ot (m\ot n)\mapsto 
h\cdot (m\ot n):=h_1\cdot m\ot h_2\cdot n$. \\
(ii) If $(M, \alpha _M)$ and 
$(N, \alpha _N)$ are left $H$-comodules, with coactions denoted by 
$M\rightarrow H\ot M$, $m\mapsto m_{(-1)}\ot m_{(0)}$ and 
$N\rightarrow H\ot N$, $n\mapsto n_{(-1)}\ot n_{(0)}$, 
then $(M\ot N, \alpha _M\ot \alpha _N)$ is also a 
left $H$-comodule, with $H$-coaction $M\ot N\rightarrow H\ot (M\ot N)$, 
$m\ot n\mapsto m_{(-1)}n_{(-1)}\ot (m_{(0)}\ot n_{(0)})$. 
\end{proposition}
\begin{definition} (\cite{homquantum3}) \label{defmodcomod}
(i) Let $(A, \mu _A)$ be an associative algebra, $\alpha _A:A\rightarrow A$ an 
algebra endomorphism, $M$ a left $A$-module with action $A\ot M\rightarrow M$, $a\ot m
\mapsto a\cdot m$, and $\alpha _M:M\rightarrow M$ a linear map satisfying the condition 
$\alpha _M(a\cdot m)=\alpha _A(a)\cdot \alpha _M(m)$, for all $a\in A$, $m\in M$. 
Then $(M, \alpha _M)$ becomes a left module over the Hom-associative algebra 
$A_{\alpha _A}$, with action  $A_{\alpha _A}\ot M\rightarrow M$, $a\ot m
\mapsto a\triangleright m:=\alpha _M(a\cdot m)=\alpha _A(a)\cdot \alpha _M(m)$. \\
(ii) Let $(C, \Delta _C)$ be a coassociative coalgebra, $\alpha _C:C\rightarrow C$ a 
coalgebra endomorphism, $M$ a left $C$-comodule with structure 
$M\rightarrow C\ot M$, $m\mapsto m_{(-1)}\ot m_{(0)}$, and 
$\alpha _M:M\rightarrow M$ a linear map satisfying the condition 
$\alpha _M(m)_{(-1)}\ot \alpha _M(m)_{(0)}=\alpha _C(m_{(-1)})\ot \alpha _M(m_{(0)})$, 
for all $m\in M$. Then $(M, \alpha _M)$ becomes a left comodule over the Hom-coassociative 
coalgebra $C_{\alpha _C}$, with coaction $M\rightarrow C_{\alpha _C}\ot M$, 
$m\mapsto m_{<-1>}\ot m_{<0>}:=\alpha _M(m)_{(-1)}\ot \alpha _M(m)_{(0)}=
\alpha _C(m_{(-1)})\ot \alpha _M(m_{(0)})$.
\end{definition}
\begin{definition} (\cite{homquantum1}, \cite{homquantum2}) 
Let $(H, \mu , \Delta, \alpha )$ be a Hom-bialgebra and $R\in H\ot H$ an element, with 
Sweedler-type notation $R=R^1\ot R^2=r^1\ot r^2$. We call $(H, \mu , \Delta, \alpha , R)$ 
a quasitriangular Hom-bialgebra if the following axioms are satisfied: 
\begin{eqnarray}
&&(\Delta \ot \alpha )(R)=\alpha (R^1)\ot \alpha (r^1)\ot R^2r^2, \label{homQT1} \\
&&(\alpha \ot \Delta )(R)=R^1r^1\ot \alpha (r^2)\ot \alpha (R^2), \label{homQT2} \\
&&\Delta ^{cop}(h)R=R\Delta (h), \label{homQT3}
\end{eqnarray}
for all $h\in H$, where we denoted as usual $\Delta ^{cop }(h)=h_2\ot h_1$. 
\end{definition}
\begin{definition} (\cite{homquantum2}) 
Let $(H, \mu , \Delta, \alpha )$ be a Hom-bialgebra and $\sigma : H\ot H\rightarrow k$ a linear map.
We call $(H, \mu , \Delta, \alpha , \sigma )$ 
a coquasitriangular Hom-bialgebra if, for all $x, y, z\in H$, we have:
\begin{eqnarray}
&&\sigma (xy\ot \alpha (z))=\sigma (\alpha (x)\ot z_1)
\sigma (\alpha (y)\ot z_2), \\
&&\sigma (\alpha (x)\ot yz)=\sigma (x_1\ot \alpha (z))\sigma (x_2\ot \alpha (y)), \\
&&y_1x_1\sigma (x_2\ot y_2)=\sigma (x_1\ot y_1)x_2y_2. 
\end{eqnarray}
\end{definition}
\section{Yetter-Drinfeld modules}
\setcounter{equation}{0}
${\;\;\;}$We introduce in this section the concept of Yetter-Drinfeld module over a Hom-bialgebra. 
We study the category of Yetter-Drinfeld modules for which the structure map is bijective and such 
that the Hom-bialgebra structure map is bijective as well.
\begin{definition}
Let $(H, \mu _H, \Delta _H, \alpha _H)$ be a Hom-bialgebra, $M$ a linear space and 
$\alpha _M:M\rightarrow M$ a linear map such that $(M, \alpha _M)$ is a left $H$-module with 
action $H\ot M\rightarrow M$, $h\ot m\mapsto h\cdot m$ and a 
left $H$-comodule with coaction $M\rightarrow H\ot M$, $m\mapsto m_{(-1)}\ot m_{(0)}$. 
Then $(M, \alpha _M)$ is called a (left-left) {\em Yetter-Drinfeld module} over $H$  if the 
following identity holds, for all $h\in H$, $m\in M$:
\begin{eqnarray}
&&(h_1\cdot m)_{(-1)}\alpha _H^2(h_2)\ot (h_1\cdot m)_{(0)}=
\alpha _H^2(h_1)\alpha _H(m_{(-1)})\ot \alpha _H(h_2)\cdot m_{(0)}. 
\label{homYD}
\end{eqnarray}
\end{definition}
\begin{definition}
Let $(H, \mu _H, \Delta _H, \alpha _H)$ be a Hom-bialgebra such that $\alpha _H$ is bijective. 
We denote by $_H^H{\mathcal YD}$ the category whose objects are Yetter-Drinfeld modules 
$(M, \alpha _M)$ over $H$, with $\alpha _M$ bijective; the morphisms in the category are morphisms 
of left $H$-modules and left $H$-comodules.
\end{definition}

The choice of the compatibility condition (\ref{homYD}) is motivated by the following result:
\begin{proposition}\label{motivYD}
Let $(H, \mu _H, \Delta _H)$ be a bialgebra, $\alpha _H:H\rightarrow H$ a bialgebra endomorphism, 
$M$ a Yetter-Drinfeld module over $H$ with notation as in Definition \ref{defYD}, 
$\alpha _M:M\rightarrow M$ a linear map satisfying the conditions in Proposition \ref{defmodcomod} 
(both in (i) and (ii)), so we can consider $(M, \alpha _M)$ both as a left $H_{\alpha _H}$-module 
with action $\triangleright $ and as a left $H_{\alpha _H}$-comodule with coaction 
$m\mapsto m_{<-1>}\ot m_{<0>}$, as in Proposition \ref{defmodcomod}. Then $(M, \alpha _M)$ with 
these structures is a Yetter-Drinfeld module over the Hom-bialgebra $H_{\alpha _H}$. 
\end{proposition}
\begin{proof}
We only need to check the Yetter-Drinfeld compatibility condition (\ref{homYD}), which in this case reads
\begin{eqnarray*}
&&(h_{(1)}\triangleright m)_{<-1>}*\alpha _H^2(h_{(2)})\ot (h_{(1)}\triangleright m)_{<0>}=
\alpha _H^2(h_{(1)})*\alpha _H(m_{<-1>})\ot \alpha _H(h_{(2)})\triangleright m_{<0>}, 
\end{eqnarray*}
where we denoted by $h*h'=\alpha _H(hh')$ and $h_{(1)}\ot h_{(2)}=\alpha _H(h_1)\ot \alpha _H(h_2)$ 
the multiplication and comultiplication of $H_{\alpha _H}$. Now we compute:\\[2mm]
${\;\;\;\;}$$(h_{(1)}\triangleright m)_{<-1>}*\alpha _H^2(h_{(2)})\ot (h_{(1)}\triangleright m)_{<0>}$
\begin{eqnarray*}
&=&\alpha _H((\alpha _H(h_1)\triangleright m)_{<-1>}\alpha _H^3(h_2))\ot 
(\alpha _H(h_1)\triangleright m)_{<0>}\\
&=&\alpha _H(\alpha _H((\alpha _H(h_1)\triangleright m)_{(-1)})\alpha _H^3(h_2))\ot 
\alpha _M((\alpha _H(h_1)\triangleright m)_{(0)})\\
&=&\alpha _H^2((\alpha _H^2(h_1)\cdot \alpha _M(m))_{(-1)}\alpha _H^2(h_2))\ot 
\alpha _M((\alpha _H^2(h_1)\cdot \alpha _M(m))_{(0)})\\
&=&\alpha _H^2((\alpha _H^2(h)_1\cdot \alpha _M(m))_{(-1)}\alpha _H^2(h)_2)\ot 
\alpha _M((\alpha _H^2(h)_1\cdot \alpha _M(m))_{(0)})\\
&\overset{(\ref{asocYD})}{=}&\alpha _H^2(\alpha _H^2(h)_1\alpha _M(m)_{(-1)})\ot 
\alpha _M(\alpha _H^2(h)_2\cdot \alpha _M(m)_{(0)})\\
&=&\alpha _H^2(\alpha _H^2(h_1)m_{<-1>})\ot 
\alpha _M(\alpha _H^2(h_2)\cdot m_{<0>})\\
&=&\alpha _H(\alpha _H^2(\alpha _H(h_1))\alpha _H(m_{<-1>}))\ot 
\alpha _H(\alpha _H(h_2))\triangleright m_{<0>}\\
&=&\alpha _H(\alpha _H^2(h_{(1)})\alpha _H(m_{<-1>}))\ot 
\alpha _H(h_{(2)})\triangleright m_{<0>}\\
&=&\alpha _H^2(h_{(1)})*\alpha _H(m_{<-1>})\ot \alpha _H(h_{(2)})\triangleright m_{<0>},
\end{eqnarray*}
finishing the proof.
\end{proof}
\begin{proposition} \label{BYD}
Let $(H, \mu _H, \Delta _H, \alpha _H)$ be a Hom-bialgebra with $\alpha _H$ bijective, 
$(M, \alpha _M)$, $(N, \alpha _N)$ two Yetter-Drinfeld modules over $H$, with notation 
as above, and define the linear map 
\begin{eqnarray}
&&B_{M, N}:M\ot N\rightarrow N\ot M, \;\;\;B_{M, N}(m\ot n)=\alpha _H^{-1}(m_{(-1)})\cdot n
\ot m_{(0)}. \label{defB}
\end{eqnarray}
Then, we have $(\alpha _N\ot \alpha _M)\circ B_{M, N}=B_{M, N}\circ 
(\alpha _M\ot \alpha _N)$ and, 
if $(P, \alpha _P)$ is another Yetter-Drinfeld module over $H$, the maps $B_{-, -}$ 
satisfy the Hom-Yang-Baxter equation (HYBE):
\begin{eqnarray}
&&(\alpha _P\ot B_{M, N})\circ (B_{M, P}\ot \alpha _N)\circ (\alpha _M\ot B_{N, P})\nonumber \\
&&\;\;\;\;\;\;\;\;\;\;\;\;\;=(B_{N, P}\ot \alpha _M)\circ (\alpha _N\ot B_{M, P})
\circ (B_{M, N}\ot \alpha _P). \label{YDhombraid}
\end{eqnarray}
\end{proposition}
\begin{proof} The condition $(\alpha _N\ot \alpha _M)\circ B_{M, N}=B_{M, N}\circ 
(\alpha _M\ot \alpha _N)$ is very easy to prove and is left to the reader. Now 
we compute:\\[2mm]
${\;\;\;\;}$$((\alpha _P\ot B_{M, N})\circ (B_{M, P}\ot \alpha _N)\circ (\alpha _M\ot B_{N, P}))
(m\ot n\ot p)$
\begin{eqnarray*}
&=&((\alpha _P\ot B_{M, N})\circ (B_{M, P}\ot \alpha _N))(\alpha _M(m)\ot 
\alpha _H^{-1}(n_{(-1)})\cdot p\ot n_{(0)})\\
&=&(\alpha _P\ot B_{M, N})(\alpha _H^{-1}(\alpha _M(m)_{(-1)})\cdot 
(\alpha _H^{-1}(n_{(-1)})\cdot p)\ot \alpha _M(m)_{(0)}\ot \alpha _N(n_{(0)}))\\
&\overset{(\ref{leftcom1})}{=}&(\alpha _P\ot B_{M, N})(m_{(-1)}\cdot 
(\alpha _H^{-1}(n_{(-1)})\cdot p)\ot \alpha _M(m_{(0)})\ot \alpha _N(n_{(0)}))\\
&\overset{(\ref{hommod2})}{=}&(\alpha _P\ot B_{M, N})(\alpha _H^{-1}(m_{(-1)}n_{(-1)})\cdot 
\alpha _P(p)\ot \alpha _M(m_{(0)})\ot \alpha _N(n_{(0)}))\\
&\overset{(\ref{hommod1})}{=}&(m_{(-1)}n_{(-1)})\cdot 
\alpha _P^2(p)\ot \alpha _H^{-1}(\alpha _M(m_{(0)})_{(-1)})\cdot \alpha _N(n_{(0)})\ot 
\alpha _M(m_{(0)})_{(0)}\\
&\overset{(\ref{leftcom1})}{=}&(m_{(-1)}n_{(-1)})\cdot 
\alpha _P^2(p)\ot m_{(0)_{(-1)}}\cdot \alpha _N(n_{(0)})\ot 
\alpha _M(m_{(0)_{(0)}}), 
\end{eqnarray*}
${\;\;\;\;}$$((B_{N, P}\ot \alpha _M)\circ (\alpha _N\ot B_{M, P})
\circ (B_{M, N}\ot \alpha _P))(m\ot n\ot p)$
\begin{eqnarray*}
&=&((B_{N, P}\ot \alpha _M)\circ (\alpha _N\ot B_{M, P}))(\alpha _H^{-1}(m_{(-1)})\cdot n
\ot m_{(0)}\ot \alpha _P(p))\\
&=&(B_{N, P}\ot \alpha _M)(\alpha _N(\alpha _H^{-1}(m_{(-1)})\cdot n)\ot 
\alpha _H^{-1}(m_{(0)_{(-1)}})\cdot \alpha _P(p)\ot m_{(0)_{(0)}})\\
&\overset{(\ref{hommod1})}{=}&(B_{N, P}\ot \alpha _M)(m_{(-1)}\cdot \alpha _N(n)\ot 
\alpha _H^{-1}(m_{(0)_{(-1)}})\cdot \alpha _P(p)\ot m_{(0)_{(0)}})\\
&=&\alpha _H^{-1}((m_{(-1)}\cdot \alpha _N(n))_{(-1)})\cdot 
(\alpha _H^{-1}(m_{(0)_{(-1)}})\cdot \alpha _P(p))\\
&&\ot (m_{(-1)}\cdot \alpha _N(n))_{(0)}
\ot \alpha _M(m_{(0)_{(0)}})\\
&\overset{(\ref{hommod2})}{=}&[\alpha _H^{-2}((m_{(-1)}\cdot \alpha _N(n))_{(-1)})
\alpha _H^{-1}(m_{(0)_{(-1)}})]\cdot \alpha _P^2(p)\\
&&\ot (m_{(-1)}\cdot \alpha _N(n))_{(0)}
\ot \alpha _M(m_{(0)_{(0)}})\\
&\overset{(\ref{leftcom2})}{=}&\alpha _H^{-2}((\alpha _H^{-1}(m_{(-1)_1})\cdot \alpha _N(n))_{(-1)}
\alpha _H(m_{(-1)_2}))\cdot \alpha _P^2(p)\\
&&\ot (\alpha _H^{-1}(m_{(-1)_1})\cdot \alpha _N(n))_{(0)}
\ot \alpha _M^2(m_{(0)})\\
&\overset{(\ref{hombia3})}{=}&\alpha _H^{-2}((\alpha _H^{-1}(m_{(-1)})_1\cdot \alpha _N(n))_{(-1)}
\alpha _H^2(\alpha _H^{-1}(m_{(-1)})_2))\cdot \alpha _P^2(p)\\
&&\ot (\alpha _H^{-1}(m_{(-1)})_1\cdot \alpha _N(n))_{(0)}
\ot \alpha _M^2(m_{(0)})\\
&\overset{(\ref{homYD})}{=}&\alpha _H^{-2}(\alpha _H^2(\alpha _H^{-1}(m_{(-1)})_1)
\alpha _H(\alpha _N(n)_{(-1)}))
\cdot \alpha _P^2(p)\\
&&\ot \alpha _H(\alpha _H^{-1}(m_{(-1)})_2)\cdot \alpha _N(n)_{(0)}
\ot \alpha _M^2(m_{(0)})\\
&\overset{(\ref{hombia3})}{=}&(\alpha _H^{-1}(m_{(-1)_1})
\alpha _H^{-1}(\alpha _N(n)_{(-1)}))
\cdot \alpha _P^2(p)\ot m_{(-1)_2}\cdot \alpha _N(n)_{(0)}
\ot \alpha _M^2(m_{(0)})\\
&\overset{(\ref{leftcom1}), \;(\ref{leftcom2})}{=}&(m_{(-1)}n_{(-1)})\cdot 
\alpha _P^2(p)\ot m_{(0)_{(-1)}}\cdot \alpha _N(n_{(0)})\ot 
\alpha _M(m_{(0)_{(0)}}), 
\end{eqnarray*}
and the two terms are obviously equal.
\end{proof}

Let now $(H, \mu _H, \Delta _H)$ be a bialgebra, $M$ a Yetter-Drinfeld module over $H$ with 
notation as in Definition \ref{defYD} and $\alpha _M:M\rightarrow M$ a morphism of 
left $H$-modules and left $H$-comodules. If we consider the map $\alpha _H:=id_H$, then one can 
easily see that the hypotheses of Proposition \ref{motivYD} are satisfied. So, $(M, \alpha _M)$ is 
a Yetter-Drinfeld module over the Hom-bialgebra $H_{id_H}$ (which is actually the bialgebra $H$). 
For this Yetter-Drinfeld module $(M, \alpha _M)$ we apply Proposition \ref{BYD}: it follows that 
the linear map $B:M\ot M\rightarrow M\ot M$, $B(m\ot m')=m_{(-1)}\cdot m'\ot m_{(0)}$, for 
$m, m'\in M$, satisfies the HYBE $(\alpha _M\ot B)\circ (B\ot \alpha _M)\circ 
(\alpha _M\ot B)=(B\ot \alpha _M)\circ (\alpha _M\ot B)\circ (B\ot \alpha _M)$. This is 
exactly the content of Theorem 4.1 in \cite{yauhomyb2}, which may thus be seen as a particular case 
of Proposition \ref{BYD}. 
\section{The quasi-braided pre-tensor category ($_H^H{\mathcal YD}$, $\hot $, $a$, $c$)}
\setcounter{equation}{0}
${\;\;\;}$We show, in this section,  that over a Hom-bialgebra with bijective structure map, the category of 
Yetter-Drinfeld modules with bijective structure maps is a quasi-braided pre-tensor category. 
It comes with solutions to the braid relation and the HYBE.
\begin{proposition} \label{deftensprod}
Let $(H, \mu _H, \Delta _H, \alpha _H)$ be a Hom-bialgebra with $\alpha _H$ bijective, 
$(M, \alpha _M)$, $(N, \alpha _N)$ two Yetter-Drinfeld modules over $H$, with notation 
as above, and define the linear maps
\begin{eqnarray*}
&&H\ot (M\ot N)\rightarrow M\ot N, \;\;\;h\ot (m\ot n)\mapsto h_1\cdot m\ot h_2\cdot n, \\
&&M\ot N\rightarrow H\ot (M\ot N), \;\;\;m\ot n\mapsto \alpha _H^{-2}(m_{(-1)}n_{(-1)})\ot 
(m_{(0)}\ot n_{(0)}). 
\end{eqnarray*}
Then $(M\ot N, \alpha _M\ot \alpha _N)$ with these structures becomes a Yetter-Drinfeld module 
over $H$, 
denoted in what follows by $M\hot N$. 
\end{proposition}
\begin{proof}
We know from Proposition \ref{tensprodmod} that $M\hot N$ is a left $H$-module. A similar 
and straightforward computation shows that  $M\hot N$ is also a left $H$-comodule. So 
we only have to prove the Yetter-Drinfeld compatibility condition (\ref{homYD}). We compute:\\[2mm]
${\;\;\;}$$(h_1\cdot (m\ot n))_{(-1)}\alpha _H^2(h_2)\ot (h_1\cdot (m\ot n))_{(0)}$
\begin{eqnarray*}
&=&((h_1)_1\cdot m\ot (h_1)_2\cdot n)_{(-1)}\alpha _H^2(h_2)\ot 
((h_1)_1\cdot m\ot (h_1)_2\cdot n)_{(0)}\\
&\overset{(\ref{hombia1})}{=}&(\alpha _H(h_1)\cdot m\ot (h_2)_1\cdot n)_{(-1)}
\alpha _H((h_2)_2)\ot 
(\alpha _H(h_1)\cdot m\ot (h_2)_1\cdot n)_{(0)}\\
&=&\alpha _H^{-2}((\alpha _H(h_1)\cdot m)_{(-1)}((h_2)_1\cdot n)_{(-1)})
\alpha _H((h_2)_2)\ot (\alpha _H(h_1)\cdot m)_{(0)}\ot ((h_2)_1\cdot n)_{(0)}\\
&=&\alpha _H^{-2}([(\alpha _H(h_1)\cdot m)_{(-1)}((h_2)_1\cdot n)_{(-1)}]
\alpha _H^3((h_2)_2))\ot (\alpha _H(h_1)\cdot m)_{(0)}\ot ((h_2)_1\cdot n)_{(0)}\\
&=&\alpha _H^{-2}(\alpha _H((\alpha _H(h_1)\cdot m)_{(-1)})[((h_2)_1\cdot n)_{(-1)}
\alpha _H^2((h_2)_2)])\ot (\alpha _H(h_1)\cdot m)_{(0)}\ot ((h_2)_1\cdot n)_{(0)}\\
&\overset{(\ref{homYD})}{=}&\alpha _H^{-2}(\alpha _H((\alpha _H(h_1)\cdot m)_{(-1)})
[\alpha _H^2((h_2)_1)\alpha _H(n_{(-1)})])\ot (\alpha _H(h_1)\cdot m)_{(0)}
\ot \alpha _H((h_2)_2)\cdot n_{(0)}\\
&=&\alpha _H^{-2}([(\alpha _H(h_1)\cdot m)_{(-1)}
\alpha _H^2((h_2)_1)]\alpha _H^2(n_{(-1)}))\ot (\alpha _H(h_1)\cdot m)_{(0)}
\ot \alpha _H((h_2)_2)\cdot n_{(0)}\\
&\overset{(\ref{hombia1})}{=}&\alpha _H^{-2}([((h_1)_1\cdot m)_{(-1)}
\alpha _H^2((h_1)_2)]\alpha _H^2(n_{(-1)}))\ot ((h_1)_1\cdot m)_{(0)}
\ot \alpha _H^2(h_2)\cdot n_{(0)}\\
&\overset{(\ref{homYD})}{=}&\alpha _H^{-2}([\alpha _H^2((h_1)_1)\alpha _H(m_{(-1)})]
\alpha _H^2(n_{(-1)}))\ot \alpha _H((h_1)_2)\cdot m_{(0)}
\ot \alpha _H^2(h_2)\cdot n_{(0)}\\
&=&\alpha _H^{-2}(\alpha _H^3((h_1)_1)\alpha _H(m_{(-1)}n_{(-1)}))
\ot \alpha _H((h_1)_2)\cdot m_{(0)}
\ot \alpha _H^2(h_2)\cdot n_{(0)}\\
&\overset{(\ref{hombia1})}{=}&\alpha _H^2(h_1)\alpha _H(\alpha _H^{-2}(m_{(-1)}n_{(-1)}))
\ot \alpha _H((h_2)_1)\cdot m_{(0)}
\ot \alpha _H((h_2)_2)\cdot n_{(0)}\\
&=&\alpha _H^2(h_1)\alpha _H((m\ot n)_{(-1)})\ot \alpha _H(h_2)\cdot (m\ot n)_{(0)},
\end{eqnarray*}
finishing the proof.
\end{proof}
\begin{proposition} \label{assoccon}
Let $(H, \mu _H, \Delta _H, \alpha _H)$ be a Hom-bialgebra such that $\alpha _H$ is bijective and 
assume that 
$(M, \alpha _M)$, $(N, \alpha _N)$, $(P, \alpha _P)$  are three Yetter-Drinfeld modules over $H$, 
with notation 
as above, such that $\alpha _M$, $\alpha _N$, $\alpha _P$ are bijective;  
define the linear map
\begin{eqnarray*}
&&a_{M, N, P}:(M\hot N)\hot P\rightarrow M\hot (N\hot P), \;\;\;
a_{M, N, P}((m\ot n)\ot p)=\alpha _M^{-1}(m)\ot (n\ot \alpha _P(p)).
\end{eqnarray*}
Then $a_{M, N, P}$ is an isomorphism of left $H$-modules and left $H$-comodules. 
\end{proposition}
\begin{proof}
It is obvious that $a_{M, N, P}$ is bijective and satisfies the relation 
$(\alpha _M\ot \alpha _N\ot \alpha _P)\circ a_{M, N, P}=a_{M, N, P}\circ 
(\alpha _M\ot \alpha _N\ot \alpha _P)$. The $H$-linearity of $a_{M, N, P}$ follows from the 
computation performed in \cite{stef}, proof of Proposition 2.6, but we include a proof here for 
reader's convenience:
\begin{eqnarray*}
a_{M, N, P}(h\cdot ((m\ot n)\ot p))&=&
a_{M, N, P}(((h_1)_1\cdot m\ot (h_1)_2\cdot n)\ot h_2\cdot p)\\
&=&\alpha _M^{-1}((h_1)_1\cdot m)\ot ((h_1)_2\cdot n\ot \alpha _P(h_2\cdot p))\\
&\overset{(\ref{hommod1})}{=}&\alpha _H^{-1}((h_1)_1)\cdot \alpha _M^{-1}(m)
\ot ((h_1)_2\cdot n
\ot \alpha _H(h_2)\cdot \alpha _P(p))\\
&\overset{(\ref{hombia1})}{=}&h_1\cdot \alpha _M^{-1}(m)
\ot ((h_2)_1\cdot n
\ot (h_2)_2\cdot \alpha _P(p))\\
&=&h_1\cdot \alpha _M^{-1}(m)\ot h_2\cdot (n\ot \alpha _P(p))\\
&=&h\cdot a_{M, N, P}((m\ot n)\ot p), \;\;\;q.e.d.
\end{eqnarray*}
Now we prove the $H$-colinearity of $a_{M, N, P}$ (denoting by $\lambda _X$ the left 
$H$-comodule structure of a Yetter-Drinfeld module $X$):\\[2mm]
${\;\;\;}$$(id_H\ot a_{M, N, P})\circ \lambda _{(M\hot N)\hot P}((m\ot n)\ot p)$
\begin{eqnarray*}
&=&(id_H\ot a_{M, N, P})(\alpha _H^{-2}((m\ot n)_{(-1)}p_{(-1)})\ot 
(m\ot n)_{(0)}\ot p_{(0)})\\
&=&\alpha _H^{-2}(\alpha _H^{-2}(m_{(-1)}n_{(-1)})p_{(-1)})\ot 
a_{M, N, P}((m_{(0)}\ot n_{(0)})\ot p_{(0)})\\
&=&\alpha _H^{-4}(m_{(-1)}n_{(-1)})\alpha _H^{-2}(p_{(-1)})\ot 
\alpha _M^{-1}(m_{(0)})\ot (n_{(0)}\ot \alpha _P(p_{(0)})), 
\end{eqnarray*}
${\;\;\;}$$(\lambda _{M\hot (N\hot P)}\circ a_{M, N, P})((m\ot n)\ot p)$
\begin{eqnarray*}
&=&\lambda _{M\hot (N\hot P)}(\alpha _M^{-1}(m)\ot (n\ot \alpha _P(p)))\\
&=&\alpha _H^{-2}(\alpha _M^{-1}(m)_{(-1)}(n\ot \alpha _P(p))_{(-1)})\ot 
\alpha _M^{-1}(m)_{(0)}\ot (n\ot \alpha _P(p))_{(0)}\\
&=&\alpha _H^{-2}(\alpha _M^{-1}(m)_{(-1)}\alpha _H^{-2}(n_{(-1)}\alpha _P(p)_{(-1)}))\ot 
\alpha _M^{-1}(m)_{(0)}\ot (n_{(0)}\ot \alpha _P(p)_{(0)})\\
&\overset{(\ref{leftcom1})}{=}&\alpha _H^{-3}(m_{(-1)})\alpha _H^{-4}(n_{(-1)}\alpha _H(p_{(-1)}))\ot 
\alpha _M^{-1}(m_{(0)})\ot (n_{(0)}\ot \alpha _P(p_{(0)}))\\
&=&\alpha _H^{-4}(\alpha _H(m_{(-1)})[n_{(-1)}\alpha _H(p_{(-1)})])\ot 
\alpha _M^{-1}(m_{(0)})\ot (n_{(0)}\ot \alpha _P(p_{(0)}))\\
&=&\alpha _H^{-4}((m_{(-1)}n_{(-1)})\alpha _H^2(p_{(-1)}))\ot 
\alpha _M^{-1}(m_{(0)})\ot (n_{(0)}\ot \alpha _P(p_{(0)}))\\
&=&\alpha _H^{-4}(m_{(-1)}n_{(-1)})\alpha _H^{-2}(p_{(-1)})\ot 
\alpha _M^{-1}(m_{(0)})\ot (n_{(0)}\ot \alpha _P(p_{(0)})),
\end{eqnarray*}
and the two terms are obviously equal. 
\end{proof}
\begin{proposition} \label{defbraiding}
Let $(H, \mu _H, \Delta _H, \alpha _H)$ be a Hom-bialgebra such that $\alpha _H$ is bijective, let 
$(M, \alpha _M)$ and $(N, \alpha _N)$  be two Yetter-Drinfeld modules over $H$, 
with notation 
as above, such that $\alpha _M$ and $\alpha _N$ are bijective, 
and define the linear map
\begin{eqnarray}
&&c_{M, N}:M\hot N\rightarrow N\hot M, \;\;\;c_{M, N}(m\ot n)=
\alpha _N^{-1}(\alpha _H^{-1}(m_{(-1)})\cdot n)\ot \alpha _M^{-1}(m_{(0)}). \label{defc}
\end{eqnarray}
Then $c_{M, N}$ is a morphism of left $H$-modules and left $H$-comodules. 
\end{proposition}
\begin{proof}
The relation $(\alpha _N\ot \alpha _M)\circ c_{M, N}=c_{M, N}\circ (\alpha _M\ot \alpha _N)$ follows by 
an easy computation using (\ref{leftcom1}) and (\ref{hommod1}). We prove now the $H$-linearity of 
$c_{M, N}$: 
\begin{eqnarray*}
c_{M, N}(h\cdot (m\ot n))&=&c_{M, N}(h_1\cdot m\ot h_2\cdot n)\\
&=&\alpha _N^{-1}(\alpha _H^{-1}((h_1\cdot m)_{(-1)})\cdot (h_2\cdot n))\ot 
\alpha _M^{-1}((h_1\cdot m)_{(0)})\\
&\overset{(\ref{hommod2})}{=}&\alpha _N^{-1}([\alpha _H^{-2}((h_1\cdot m)_{(-1)})h_2]\cdot 
\alpha _N(n))\ot 
\alpha _M^{-1}((h_1\cdot m)_{(0)})\\
&=&\alpha _N^{-1}(\alpha _H^{-2}((h_1\cdot m)_{(-1)}\alpha _H^2(h_2))\cdot 
\alpha _N(n))\ot 
\alpha _M^{-1}((h_1\cdot m)_{(0)})\\
&\overset{(\ref{homYD})}{=}&\alpha _N^{-1}(\alpha _H^{-2}(\alpha _H^2(h_1)\alpha _H(m_{(-1)}))\cdot 
\alpha _N(n))\ot 
\alpha _M^{-1}(\alpha _H(h_2)\cdot m_{(0)})\\
&\overset{(\ref{hommod2})}{=}&\alpha _N^{-1}(\alpha _H(h_1)\cdot (\alpha _H^{-1}(m_{(-1)})\cdot n))\ot 
\alpha _M^{-1}(\alpha _H(h_2)\cdot m_{(0)})\\
&\overset{(\ref{hommod1})}{=}&h_1\cdot \alpha _N^{-1}(\alpha _H^{-1}(m_{(-1)})\cdot n)\ot 
h_2\cdot \alpha _M^{-1}(m_{(0)})\\
&=&h\cdot c_{M, N}(m\ot n), \;\;\;q.e.d.
\end{eqnarray*}
Now we prove the $H$-colinearity of $c_{M, N}$ (we denote by $\lambda _{M\hot N}$ and 
$\lambda _{N\hot M}$ the 
left $H$-comodule structures of $M\hot N$ and respectively $N\hot M$): \\[2mm]
${\;\;\;}$$(\lambda _{N\hot M}\circ c_{M, N})(m\ot n)$
\begin{eqnarray*}
&=&\lambda _{N\hot M}(\alpha _N^{-1}(\alpha _H^{-1}(m_{(-1)})\cdot n)\ot \alpha _M^{-1}(m_{(0)}))\\
&=&\alpha _H^{-2}(\alpha _N^{-1}(\alpha _H^{-1}(m_{(-1)})\cdot n)_{(-1)}\alpha _M^{-1}(m_{(0)})_{(-1)})
\ot \alpha _N^{-1}(\alpha _H^{-1}(m_{(-1)})\cdot n)_{(0)}\ot \alpha _M^{-1}(m_{(0)})_{(0)}\\
&\overset{(\ref{leftcom1})}{=}&\alpha _H^{-2}(\alpha _N^{-1}(\alpha _H^{-1}(m_{(-1)})\cdot n)_{(-1)}
\alpha _H^{-1}(m_{(0)_{(-1)}}))
\ot \alpha _N^{-1}(\alpha _H^{-1}(m_{(-1)})\cdot n)_{(0)}\ot \alpha _M^{-1}(m_{(0)_{(0)}})\\
&\overset{(\ref{leftcom2})}{=}&\alpha _H^{-2}(\alpha _N^{-1}(\alpha _H^{-2}(m_{(-1)_1})\cdot n)_{(-1)}
\alpha _H^{-1}(m_{(-1)_2}))
\ot \alpha _N^{-1}(\alpha _H^{-2}(m_{(-1)_1})\cdot n)_{(0)}\ot m_{(0)}\\
&\overset{(\ref{leftcom1})}{=}&\alpha _H^{-3}((\alpha _H^{-2}(m_{(-1)_1})\cdot n)_{(-1)}
m_{(-1)_2})
\ot \alpha _N^{-1}((\alpha _H^{-2}(m_{(-1)_1})\cdot n)_{(0)})\ot m_{(0)}\\
&=&\alpha _H^{-3}((\alpha _H^{-2}(m_{(-1)_1})\cdot n)_{(-1)}
\alpha _H^2(\alpha _H^{-2}(m_{(-1)_2})))
\ot \alpha _N^{-1}((\alpha _H^{-2}(m_{(-1)_1})\cdot n)_{(0)})\ot m_{(0)}\\
&\overset{(\ref{hombia3})}{=}&\alpha _H^{-3}((\alpha _H^{-2}(m_{(-1)})_1\cdot n)_{(-1)}
\alpha _H^2(\alpha _H^{-2}(m_{(-1)})_2))
\ot \alpha _N^{-1}((\alpha _H^{-2}(m_{(-1)})_1\cdot n)_{(0)})\ot m_{(0)}\\
&\overset{(\ref{homYD})}{=}&\alpha _H^{-3}(\alpha _H^2(\alpha _H^{-2}(m_{(-1)})_1)
\alpha _H(n_{(-1)}))
\ot \alpha _N^{-1}(\alpha _H(\alpha _H^{-2}(m_{(-1)})_2)\cdot n_{(0)})\ot m_{(0)}\\
&\overset{(\ref{hombia3})}{=}&\alpha _H^{-3}(m_{(-1)_1}\alpha _H(n_{(-1)}))\ot 
\alpha _N^{-1}(\alpha _H^{-1}(m_{(-1)_2})\cdot n_{(0)})\ot m_{(0)}\\
&\overset{(\ref{leftcom2})}{=}&\alpha _H^{-2}(m_{(-1)}n_{(-1)})\ot 
\alpha _N^{-1}(\alpha _H^{-1}(m_{(0)_{(-1)}})\cdot n_{(0)})\ot \alpha _M^{-1}(m_{(0)_{(0)}})\\
&=&(id_H\ot c_{M, N})(\alpha _H^{-2}(m_{(-1)}n_{(-1)})\ot 
(m_{(0)}\ot n_{(0)}))\\
&=&(id_H\ot c_{M, N})((m\ot n)_{(-1)}\ot (m\ot n)_{(0)})\\
&=&((id_H\ot c_{M, N})\circ \lambda _{M\hot N})(m\ot n), 
\end{eqnarray*}
finishing the proof.
\end{proof}
\begin{theorem}
Let $(H, \mu _H, \Delta _H, \alpha _H)$ be a Hom-bialgebra such that $\alpha _H$ is bijective. 
Then $_H^H{\mathcal YD}$ is a quasi-braided pre-tensor category, with tensor product 
$\hot $, associativity constraints $a_{M, N, P}$ and quasi-braiding $c_{M, N}$ defined in 
Propositions \ref{deftensprod}, \ref{assoccon} and \ref{defbraiding}, respectively. 
\end{theorem}
\begin{proof}
The only nontrivial things left to prove are the pentagon axiom for $a_{M, N, P}$ and the 
two hexagonal relations for $c_{M, N}$. The pentagon axiom for $a_{M, N, P}$ follows by 
a straightforward computation that shows the equality \\[2mm]
${\;\;\;\;}$$((id_M\ot a_{N, P, Q})\circ a_{M, N\hot P, Q}\circ (a_{M, N, P}\ot id_Q))
(((m\ot n)\ot p)\ot q)$
\begin{eqnarray*}
&=&(a_{M, N, P\hot Q}\circ a_{M\hot N, P, Q})(((m\ot n)\ot p)\ot q)\\
&=&\alpha _M^{-2}(m)\ot \alpha _N^{-1}(n)\ot \alpha _P(p)\ot \alpha _Q^2(q),
\end{eqnarray*}
for any objects $(M, \alpha_M), (N, \alpha _N), (P, \alpha _P), (Q, \alpha _Q)\in \;_H^H{\mathcal YD}$.

We prove the first hexagonal relation for $c_{M, N}$. Let 
$(M, \alpha_M), (N, \alpha _N), (P, \alpha _P)\in \;_H^H{\mathcal YD}$; 
we compute:\\[2mm]
${\;\;\;\;}$$(a_{N, P, M}\circ c_{M, N\hot P}\circ a_{M, N, P})((m\ot n)\ot p)$
\begin{eqnarray*}
&=&(a_{N, P, M}\circ c_{M, N\hot P})(\alpha _M^{-1}(m)\ot (n\ot \alpha _P(p)))\\
&=&a_{N, P, M}((\alpha _N^{-1}\ot \alpha _P^{-1})(\alpha _H^{-1}(\alpha _M^{-1}(m)_{(-1)})
\cdot (n\ot \alpha _P(p)))\ot \alpha _M^{-1}(\alpha _M^{-1}(m)_{(0)}))\\
&\overset{(\ref{leftcom1})}{=}&a_{N, P, M}([\alpha _N^{-1}(\alpha _H^{-2}(m_{(-1)})_1\cdot n)
\ot \alpha _P^{-1}(\alpha _H^{-2}(m_{(-1)})_2\cdot \alpha _P(p))]\ot \alpha _M^{-2}(m_{(0)}))\\
&=&\alpha _N^{-2}(\alpha _H^{-2}(m_{(-1)})_1\cdot n)
\ot \alpha _P^{-1}(\alpha _H^{-2}(m_{(-1)})_2\cdot \alpha _P(p))\ot \alpha _M^{-1}(m_{(0)}), 
\end{eqnarray*}
${\;\;\;\;}$$((id_N\ot c_{M, P})\circ a_{N, M, P}\circ (c_{M, N}\ot id_P))((m\ot n)\ot p)$
\begin{eqnarray*}
&=&((id_N\ot c_{M, P})\circ a_{N, M, P})((\alpha _N^{-1}(\alpha _H^{-1}(m_{(-1)})\cdot n)\ot 
\alpha _M^{-1}(m_{(0)}))\ot p)\\
&=&(id_N\ot c_{M, P})(\alpha _N^{-2}(\alpha _H^{-1}(m_{(-1)})\cdot n)\ot 
(\alpha _M^{-1}(m_{(0)})\ot \alpha _P(p)))\\
&=&\alpha _N^{-2}(\alpha _H^{-1}(m_{(-1)})\cdot n)\ot 
\alpha _P^{-1}(\alpha _H^{-1}(\alpha _M^{-1}(m_{(0)})_{(-1)})\cdot \alpha _P(p))\ot 
\alpha _M^{-1}(\alpha _M^{-1}(m_{(0)})_{(0)})\\
&\overset{(\ref{leftcom1})}{=}&\alpha _N^{-2}(\alpha _H^{-1}(m_{(-1)})\cdot n)\ot 
\alpha _P^{-1}(\alpha _H^{-2}(m_{(0)_{(-1)}})\cdot \alpha _P(p))\ot 
\alpha _M^{-2}(m_{(0)_{(0)}})\\
&\overset{(\ref{leftcom2})}{=}&\alpha _N^{-2}(\alpha _H^{-2}(m_{(-1)_1})\cdot n)\ot 
\alpha _P^{-1}(\alpha _H^{-2}(m_{(-1)_2})\cdot \alpha _P(p))\ot 
\alpha _M^{-1}(m_{(0)})\\
&\overset{(\ref{hombia3})}{=}&\alpha _N^{-2}(\alpha _H^{-2}(m_{(-1)})_1\cdot n)
\ot \alpha _P^{-1}(\alpha _H^{-2}(m_{(-1)})_2\cdot \alpha _P(p))\ot \alpha _M^{-1}(m_{(0)}),
\end{eqnarray*}
and the two terms are obviously equal.

Now we prove the second hexagonal relation for $c_{M, N}$:\\[2mm]
${\;\;\;\;}$$(a_{P, M, N}^{-1}\circ c_{M\hot N, P}\circ a_{M, N, P}^{-1})(m\ot (n\ot p))$
\begin{eqnarray*}
&=&(a_{P, M, N}^{-1}\circ c_{M\hot N, P})((\alpha _M(m)\ot n)\ot \alpha _P^{-1}(p))\\
&=&a_{P, M, N}^{-1}(\alpha _P^{-1}(\alpha _H^{-1}((\alpha _M(m)\ot n)_{(-1)})
\cdot \alpha _P^{-1}(p))\ot (\alpha _M^{-1}\ot \alpha _N^{-1})((\alpha _M(m)\ot n)_{(0)}))\\ 
&=&a_{P, M, N}^{-1}(\alpha _P^{-1}(\alpha _H^{-3}(\alpha _M(m)_{(-1)}n_{(-1)})
\cdot \alpha _P^{-1}(p))\ot (\alpha _M^{-1}(\alpha _M(m)_{(0)})\ot \alpha _N^{-1}(n_{(0)})))\\ 
&\overset{(\ref{leftcom1})}{=}&\alpha _H^{-3}(\alpha _H(m_{(-1)})n_{(-1)})
\cdot \alpha _P^{-1}(p)\ot m_{(0)}\ot \alpha _N^{-2}(n_{(0)}), 
\end{eqnarray*}
${\;\;\;\;}$$((c_{M, P}\ot id_N)\circ a_{M, P, N}^{-1}\circ (id_M\ot c_{N, P}))(m\ot (n\ot p))$
\begin{eqnarray*}
&=&((c_{M, P}\ot id_N)\circ a_{M, P, N}^{-1})(m\ot [\alpha _P^{-1}(\alpha _H^{-1}(n_{(-1)})\cdot p)
\ot \alpha _N^{-1}(n_{(0)})])\\
&=&(c_{M, P}\ot id_N)([\alpha _M(m)\ot \alpha _P^{-1}(\alpha _H^{-1}(n_{(-1)})\cdot p)]
\ot \alpha _N^{-2}(n_{(0)}))\\
&=&\alpha _P^{-1}(\alpha _H^{-1}(\alpha _M(m)_{(-1)})\cdot 
\alpha _P^{-1}(\alpha _H^{-1}(n_{(-1)})\cdot p))\ot \alpha _M^{-1}(\alpha _M(m)_{(0)})
\ot \alpha _N^{-2}(n_{(0)})\\
&\overset{(\ref{leftcom1})}{=}&\alpha _P^{-1}(m_{(-1)}\cdot 
\alpha _P^{-1}(\alpha _H^{-1}(n_{(-1)})\cdot p))\ot m_{(0)}
\ot \alpha _N^{-2}(n_{(0)})\\
&\overset{(\ref{hommod1})}{=}&\alpha _P^{-1}(m_{(-1)}\cdot 
[\alpha _H^{-2}(n_{(-1)})\cdot \alpha _P^{-1}(p)])\ot m_{(0)}
\ot \alpha _N^{-2}(n_{(0)})\\
&\overset{(\ref{hommod2})}{=}&\alpha _P^{-1}([\alpha _H^{-1}(m_{(-1)}) 
\alpha _H^{-2}(n_{(-1)})]\cdot p)\ot m_{(0)}
\ot \alpha _N^{-2}(n_{(0)})\\
&\overset{(\ref{hommod1})}{=}&\alpha _H^{-1}(\alpha _H^{-1}(m_{(-1)}) 
\alpha _H^{-2}(n_{(-1)}))\cdot \alpha _P^{-1}(p)\ot m_{(0)}
\ot \alpha _N^{-2}(n_{(0)})\\
&=&\alpha _H^{-3}(\alpha _H(m_{(-1)})n_{(-1)})
\cdot \alpha _P^{-1}(p)\ot m_{(0)}\ot \alpha _N^{-2}(n_{(0)}),
\end{eqnarray*}
finishing the proof. 
\end{proof}
\begin{theorem}
Let $(H, \mu _H, \Delta _H, \alpha _H)$ be a Hom-bialgebra such that $\alpha _H$ is bijective and  
$(M, \alpha _M)$, $(N, \alpha _N)$, $(P, \alpha _P)$ three objects in $_H^H{\mathcal YD}$. Then the 
quasi-braiding $c$ satisfies the braid relation 
\begin{eqnarray}
&&(id_P\ot c_{M, N})\circ (c_{M, P}\ot id_N)\circ (id_M\ot c_{N, P})\nonumber \\
&&\;\;\;\;\;\;\;\;\;\;\;\;\;=(c_{N, P}\ot id_M)\circ (id_N\ot c_{M, P})
\circ (c_{M, N}\ot id_P). \label{braidingbraid}
\end{eqnarray}
\end{theorem}
\begin{proof}
Since $c$ is the quasi-braiding of the pre-tensor category $_H^H{\mathcal YD}$, 
whose associators are nontrivial, it follows that $c$ satisfies the dodecagonal braid relation 
(see \cite{kassel}, p. 317)
\begin{eqnarray*}
&&(id_P\ot c_{M, N})\circ a_{P, M, N}\circ (c_{M, P}\ot id_N)\circ 
a_{M, P, N}^{-1}\circ (id_M\ot c_{N, P})\circ a_{M, N, P}\\
&&\;\;\;\;\;\;\;\;=a_{P, N, M}\circ (c_{N, P}\ot id_M)\circ a_{N, P, M}^{-1}\circ (id_N\ot c_{M, P})
\circ a_{N, M, P}\circ (c_{M, N}\ot id_P), 
\end{eqnarray*}
which, by using the formulae of the associators, becomes 
\begin{eqnarray*}
&&(id_P\ot c_{M, N})\circ (\alpha _P^{-1}\ot id_M\ot \alpha _N)\circ (c_{M, P}\ot id_N)\circ 
(\alpha _M\ot id_P\ot \alpha _N^{-1})\\
&&\;\;\;\;\;\;\;\;\;\;\;\;\;\;\;\;\;\;\;\circ (id_M\ot c_{N, P})\circ (\alpha _M^{-1}\ot id_N\ot \alpha _P)\\
&&\;\;\;\;\;\;\;\;=(\alpha _P^{-1}\ot id_N\ot \alpha _M)\circ (c_{N, P}\ot id_M)\circ 
(\alpha _N\ot id_P\ot \alpha _M^{-1})\circ (id_N\ot c_{M, P})\\
&&\;\;\;\;\;\;\;\;\;\;\;\;\;\;\;\;\;\;\;\;\;\;\;\;\;\;\;\circ (\alpha _N^{-1}\ot id_M\ot \alpha _P)
\circ (c_{M, N}\ot id_P).
\end{eqnarray*}
In the left hand side of this relation, $\alpha _N$ and $\alpha _N^{-1}$ cancel each other, 
as well as $\alpha _M$ and $\alpha _M^{-1}$. Similarly, in the right hand side, 
$\alpha _M$ and $\alpha _M^{-1}$ and also $\alpha _N$ and $\alpha _N^{-1}$ cancel each other. 
So, this relation becomes:
\begin{eqnarray*}
&&(id_P\ot c_{M, N})\circ (\alpha _P^{-1}\ot id_M\ot id_N)\circ (c_{M, P}\ot id_N)\circ 
(id_M\ot c_{N, P})\circ (id_M\ot id_N\ot \alpha _P)\\
&&\;\;=(\alpha _P^{-1}\ot id_N\ot id_M)\circ (c_{N, P}\ot id_M)\circ 
(id_N\ot c_{M, P})\circ (id_N\ot id_M\ot \alpha _P)
\circ (c_{M, N}\ot id_P), 
\end{eqnarray*}
which may be written as 
\begin{eqnarray*}
&&(\alpha _P^{-1}\ot id_N\ot id_M)\circ 
(id_P\ot c_{M, N})\circ (c_{M, P}\ot id_N)\circ 
(id_M\ot c_{N, P})\circ (id_M\ot id_N\ot \alpha _P)\\
&&\;\;=(\alpha _P^{-1}\ot id_N\ot id_M)\circ (c_{N, P}\ot id_M)\circ 
(id_N\ot c_{M, P})
\circ (c_{M, N}\ot id_P)\circ (id_M\ot id_N\ot \alpha _P), 
\end{eqnarray*}
which is obviously equivalent to (\ref{braidingbraid}). 
\end{proof}

Note that, for objects $(M, \alpha_M), (N, \alpha _N)\in \;_H^H{\mathcal YD}$, 
i.e. Yetter-Drinfeld modules with bijective $\alpha _M$ and $\alpha _N$, 
the maps $B_{M, N}$ defined in (\ref{defB}) and the maps $c_{M, N}$ 
defined in (\ref{defc}) are related by the formula $B_{M, N}=(\alpha _N\ot \alpha _M)\circ c_{M, N}$. 
Our next result shows that, in this case, the fact that the maps $B_{M, N}$ satisfy the HYBE is a 
consequence of the fact that the maps $c_{M, N}$ satisfy the braid relation. 
One may call the HYBE equation ''Hom-braid relation'', since it is a twisting of the braid relation.
\begin{proposition} \label{relbraidhombraid}
Let $M, N, P$ be linear spaces and $\alpha _M:M\rightarrow M$, $\alpha _N:N\rightarrow N$, 
$\alpha _P:P\rightarrow P$ linear maps satisfying the following conditions:
\begin{eqnarray}
&&(\alpha _N\ot \alpha _M)\circ c_{M, N}=c_{M, N}\circ (\alpha _M\ot \alpha _N), \label{com1} \\
&&(\alpha _P\ot \alpha _M)\circ c_{M, P}=c_{M, P}\circ (\alpha _M\ot \alpha _P), \label{com2} \\
&&(\alpha _P\ot \alpha _N)\circ c_{N, P}=c_{N, P}\circ (\alpha _N\ot \alpha _P), \label{com3} \\
&&(id_P\ot c_{M, N})\circ (c_{M, P}\ot id_N)\circ (id_M\ot c_{N, P})\nonumber \\
&&\;\;\;\;\;\;\;\;\;\;\;\;\;=(c_{N, P}\ot id_M)\circ (id_N\ot c_{M, P})
\circ (c_{M, N}\ot id_P). \label{braidc}
\end{eqnarray}
Define the maps $B_{M, N}:=(\alpha _N\ot \alpha _M)\circ c_{M, N}$, 
$B_{M, P}:=(\alpha _P\ot \alpha _M)\circ c_{M, P}$, $B_{N, P}:=(\alpha _P\ot \alpha _N)\circ c_{N, P}$. 
Then the following relations hold:
\begin{eqnarray}
&&(\alpha _N\ot \alpha _M)\circ B_{M, N}=B_{M, N}\circ (\alpha _M\ot \alpha _N),  \\
&&(\alpha _P\ot \alpha _M)\circ B_{M, P}=B_{M, P}\circ (\alpha _M\ot \alpha _P),  \\
&&(\alpha _P\ot \alpha _N)\circ B_{N, P}=B_{N, P}\circ (\alpha _N\ot \alpha _P),  \\
&&(\alpha _P\ot B_{M, N})\circ (B_{M, P}\ot \alpha _N)\circ (\alpha _M\ot B_{N, P})\nonumber \\
&&\;\;\;\;\;\;\;\;\;\;\;\;\;=(B_{N, P}\ot \alpha _M)\circ (\alpha _N\ot B_{M, P})
\circ (B_{M, N}\ot \alpha _P). \label{hombraidB}
\end{eqnarray}
\end{proposition}
\begin{proof}
The first three relations are obvious, because of (\ref{com1})-(\ref{com3}). 
We prove (\ref{hombraidB}): \\[2mm]
${\;\;\;\;}$$(\alpha _P\ot B_{M, N})\circ (B_{M, P}\ot \alpha _N)\circ (\alpha _M\ot B_{N, P})$
\begin{eqnarray*}
&=&(\alpha _P\ot \alpha _N\ot \alpha _M)\circ (id_P\ot c_{M, N})
\circ (\alpha _P\ot \alpha _M\ot \alpha _N)\circ (c_{M, P}\ot id_N)\\
&&\circ (\alpha _M\ot \alpha _P\ot \alpha _N)
\circ (id_M\ot c_{N, P})\\
&\overset{(\ref{com1}),\; (\ref{com3})}{=}&(\alpha _P\ot \alpha _N\ot \alpha _M)
\circ (\alpha _P\ot \alpha _N\ot \alpha _M)\circ (id_P\ot c_{M, N})\circ (c_{M, P}\ot id_N)\\
&&\circ (id_M\ot c_{N, P})\circ (\alpha _M\ot \alpha _N\ot \alpha _P)\\
&\overset{(\ref{braidc})}{=}&(\alpha _P\ot \alpha _N\ot \alpha _M)
\circ (\alpha _P\ot \alpha _N\ot \alpha _M)\circ (c_{N, P}\ot id_M)\circ (id_N\ot c_{M, P})\\
&&\circ (c_{M, N}\ot id_P)\circ (\alpha _M\ot \alpha _N\ot \alpha _P)\\
&\overset{(\ref{com1}),\; (\ref{com3})}{=}&(\alpha _P\ot \alpha _N\ot \alpha _M)
\circ (c_{N, P}\ot id_M)
\circ (\alpha _N\ot \alpha _P\ot \alpha _M)\circ (id_N\ot c_{M, P})\\
&&\circ (\alpha _N\ot \alpha _M\ot \alpha _P)\circ (c_{M, N}\ot id_P)\\
&=&(B_{N, P}\ot \alpha _M)\circ (\alpha _N\ot B_{M, P})
\circ (B_{M, N}\ot \alpha _P), 
\end{eqnarray*}
finishing the proof.
\end{proof}

A particular case of Proposition \ref{relbraidhombraid} is the following result given  in \cite{yauhomyb2}:
\begin{corollary}
Let $V$ be a linear space, $c:V\ot V\rightarrow V\ot V$ a linear map satisfying the braid relation 
$(id_V\ot c)\circ (c\ot id_V)\circ (id_V\ot c)=(c\ot id_V)\circ (id_V\ot c)\circ (c\ot id_V)$ and 
$\alpha :V\rightarrow V$ a linear map such that $(\alpha \ot \alpha )\circ c=c\circ (\alpha \ot \alpha )$. 
Then the linear map $B:=(\alpha \ot \alpha )\circ c:V\ot V\rightarrow V\ot V$ satisfies the relations 
$(\alpha \ot \alpha )\circ B=B\circ (\alpha \ot \alpha )$ and 
$(\alpha \ot B)\circ (B\ot \alpha )\circ (\alpha \ot B)=(B\ot \alpha )\circ (\alpha \ot B)\circ (B\ot \alpha )$.
\end{corollary}

We can make now the connection between Yetter-Drinfeld modules and modules over quasitriangular 
Hom-bialgebras. 
\begin{proposition} \label{quasitriang}
Let $(H, \mu _H, \Delta _H, \alpha _H, R)$ be a quasitriangular Hom-bialgebra such that
\begin{eqnarray}
&&(\alpha _H\ot \alpha _H)(R)=R.  \label{invar}
\end{eqnarray}
(i) Let $(M, \alpha _M)$ be a left $H$-module with action $H\ot M\rightarrow M$, 
$h\ot m\mapsto h\cdot m$. Define the linear map $\lambda _M:M\rightarrow H\ot M$, 
$\lambda _M(m)=m_{(-1)}\ot m_{(0)}:=\alpha _H(R^2)\ot R^1\cdot m$. 
Then $(M, \alpha _M)$ with these structures is a Yetter-Drinfeld module over $H$.  \\
(ii) Assume that $\alpha _H$ is bijective. Let $(N, \alpha _N)$ be another left $H$-module 
with action $H\ot N\rightarrow N$, $h\ot n\mapsto h\cdot n$, regarded as a Yetter-Drinfeld module 
as in (i), via the map $\lambda _N:N\rightarrow H\ot N$, 
$\lambda _N(n)=n_{(-1)}\ot n_{(0)}:=\alpha _H(r^2)\ot r^1\cdot n$. We regard 
$(M\ot N, \alpha _M\ot \alpha _N)$ as a left $H$-module via the standard action 
$h\cdot (m\ot n)=h_1\cdot m\ot h_2\cdot n$ and then we regard 
$(M\ot N, \alpha _M\ot \alpha _N)$ as a Yetter-Drinfeld module as in (i). Then this 
Yetter-Drinfeld module $(M\ot N, \alpha _M\ot \alpha _N)$ coincides with the 
Yetter-Drinfeld module $M\hot N$ defined as in Proposition \ref{deftensprod}.
\end{proposition}
\begin{proof}
(i) First we have to prove that $(M, \alpha _M)$ is a left $H$-comodule; 
(\ref{leftcom1}) is easy and left to the reader, we check (\ref{leftcom2}): 
\begin{eqnarray*}
(\Delta _H\ot \alpha _M)(\lambda _M(m))&=&\Delta _H(\alpha _H(R^2))\ot 
\alpha _M(R^1\cdot m)\\
&\overset{(\ref{hombia3}),\; (\ref{hommod1})}{=}&\alpha _H(R^2_1)\ot 
\alpha _H(R^2_2)\ot \alpha _H(R^1)\cdot \alpha _M(m)\\
&\overset{(\ref{homQT2})}{=}&\alpha _H^2(r^2)\ot \alpha _H^2(R^2)\ot 
(R^1r^1)\cdot \alpha _M(m)\\
&\overset{(\ref{hommod2})}{=}&\alpha _H^2(r^2)\ot \alpha _H^2(R^2)\ot 
\alpha _H(R^1)\cdot (r^1\cdot m)\\
&\overset{(\ref{invar})}{=}&\alpha _H^2(r^2)\ot \alpha _H(R^2)\ot 
R^1\cdot (r^1\cdot m)\\
&=&\alpha _H^2(r^2)\ot \lambda _M(r^1\cdot m)\\
&=& (\alpha _H\ot \lambda _M)(\lambda _M(m)), \;\;\;q.e.d.
\end{eqnarray*}
Now we check the Yetter-Drinfeld condition (\ref{homYD}):
\begin{eqnarray*}
(h_1\cdot m)_{(-1)}\alpha _H^2(h_2)\ot (h_1\cdot m)_{(0)}&=&
\alpha _H(R^2)\alpha _H^2(h_2)\ot R^1\cdot (h_1\cdot m)\\
&\overset{(\ref{invar})}{=}&\alpha _H^2(R^2)\alpha _H^2(h_2)\ot 
\alpha _H(R^1)\cdot (h_1\cdot m)\\
&\overset{(\ref{hommod2})}{=}&\alpha _H^2(R^2h_2)\ot 
(R^1h_1)\cdot \alpha _M(m)\\
&\overset{(\ref{homQT3})}{=}&\alpha _H^2(h_1R^2)\ot 
(h_2R^1)\cdot \alpha _M(m)\\
&\overset{(\ref{hommod2})}{=}&\alpha _H^2(h_1)\alpha _H^2(R^2)\ot 
\alpha _H(h_2)\cdot (R^1\cdot m)\\
&=&\alpha _H^2(h_1)\alpha _H(m_{(-1)})\ot \alpha _H(h_2)\cdot m_{(0)}, \;\;\;q.e.d.
\end{eqnarray*}
(ii) We only need to prove that the two comodule structures on $M\ot N$ coincide, 
that is, for all $m\in M$, $n\in N$, 
\begin{eqnarray*}
&&\alpha _H^{-2}(m_{(-1)}n_{(-1)})\ot (m_{(0)}\ot n_{(0)})=
\alpha _H(R^2)\ot R^1\cdot (m\ot n), 
\end{eqnarray*}
that is 
\begin{eqnarray*}
&&\alpha _H^{-2}(\alpha _H(R^2)\alpha _H(r^2))\ot (R^1\cdot m\ot r^1\cdot n)=
\alpha _H(R^2)\ot (R^1_1\cdot m\ot R^1_2\cdot n), 
\end{eqnarray*}
which, because of (\ref{invar}), is equivalent to 
\begin{eqnarray*}
&&R^2r^2\ot (\alpha _H(R^1)\cdot m\ot \alpha _H(r^1)\cdot n)=
\alpha _H(R^2)\ot (R^1_1\cdot m\ot R^1_2\cdot n), 
\end{eqnarray*}
and this is an obvious consequence of (\ref{homQT1}). 
\end{proof}

As a consequence of various results obtained so far, we also obtain the following: 
\begin{theorem}
Let $(H, \mu _H, \Delta _H, \alpha _H, R)$ be a quasitriangular Hom-bialgebra with 
$\alpha _H$ bijective and 
$(\alpha _H\ot \alpha _H)(R)=R$. Denote by $_H{\mathcal M}$ the category whose 
objects are left $H$-modules $(M, \alpha _M)$ with $\alpha _M$ bijective and morphisms are 
morphisms of left $H$-modules. Then $_H{\mathcal M}$ is a quasi-braided pre-tensor 
subcategory of $_H^H{\mathcal YD}$, with tensor product defined as 
in Proposition \ref{tensprodmod} (i), associativity constraints defined by the formula 
$a_{M, N, P}((m\ot n)\ot p)=\alpha _M^{-1}(m)\ot (n\ot \alpha _P(p))$, for $M, N, P\in \;
_H{\mathcal M}$, and quasi-braiding $c_{M, N}:M\ot N\rightarrow N\ot M$, 
$c_{M, N}(m\ot n)=\alpha _N^{-1}(R^2\cdot n)\ot \alpha _M^{-1}(R^1\cdot m)$, for all 
$M, N\in \;_H{\mathcal M}$.
\end{theorem}

We recall the following result  (\cite{homquantum1}, Theorem 4.4):
\begin{proposition} Let 
$(H, \mu _H, \Delta _H, \alpha _H, R)$ be a quasitriangular Hom-bialgebra such that  
$(\alpha _H\ot \alpha _H)(R)=R$ and $(M, \alpha _M)$ a left $H$-module. Then the linear 
map $B:M\ot M\rightarrow M\ot M$, $B(m\ot m')=R^2\cdot m'\ot R^1\cdot m$ is a 
solution of the HYBE for $(M, \alpha _M)$. 
\end{proposition}

It turns out that the particular case of this result in which $\alpha _H$ is bijective is a particular 
case of Proposition \ref{BYD}, via Proposition \ref{quasitriang}.
\section{The quasi-braided pre-tensor category ($_H^H{\mathcal YD}$, $\tot $, $\mathfrak{a}$, $c$)}
\setcounter{equation}{0}
${\;\;\;}$We have seen in the previous section that modules over quasitriangular Hom-bialgebras 
become Yetter-Drinfeld modules. Similarly, comodules over coquasitriangular Hom-bialgebras 
become Yetter-Drinfeld modules; inspired by this, we can introduce a second quasi-braided 
pre-tensor  category structure on $_H^H{\mathcal YD}$. We include these facts here for 
completeness. Each of the next results is the analogue of a result in the previous section; 
their proofs are similar to those of their analogues and are left to the reader. 
\begin{theorem} \label{mainmirror}
Let $(H, \mu _H, \Delta _H, \alpha _H)$ be a Hom-bialgebra such that $\alpha _H$ is bijective. \\
(i) Let  
$(M, \alpha _M)$ and $(N, \alpha _N)$ be two Yetter-Drinfeld modules over $H$, with notation 
as above, and define the linear maps 
\begin{eqnarray*}
&&H\ot (M\ot N)\rightarrow M\ot N, \;\;\;h\ot (m\ot n)\mapsto \alpha _H^{-2}(h_1)\cdot m
\ot \alpha _H^{-2}(h_2)\cdot n, \\
&&M\ot N\rightarrow H\ot (M\ot N), \;\;\;m\ot n\mapsto m_{(-1)}n_{(-1)}\ot 
(m_{(0)}\ot n_{(0)}). 
\end{eqnarray*}
Then $(M\ot N, \alpha _M\ot \alpha _N)$ with these structures becomes a Yetter-Drinfeld module 
over $H$, 
denoted in what follows by $M\tot N$. \\
(ii) $_H^H{\mathcal YD}$ is a quasi-braided pre-tensor category, with tensor product 
$\tot $ as in (i) and associativity constraints $\mfa _{M, N, P}$ and quasi-braiding $c_{M, N}$ defined 
as follows:
\begin{eqnarray*}
&&\mfa _{M, N, P}:(M\tot N)\tot P\rightarrow M\tot (N\tot P), \;\;\;
\mfa _{M, N, P}((m\ot n)\ot p)=\alpha _M(m)\ot (n\ot \alpha _P^{-1}(p)), \\
&&c_{M, N}:M\tot N\rightarrow N\tot M, \;\;\;c_{M, N}(m\ot n)=
\alpha _N^{-1}(\alpha _H^{-1}(m_{(-1)})\cdot n)\ot \alpha _M^{-1}(m_{(0)}).
\end{eqnarray*}
\end{theorem}
\begin{proposition} \label{connectcoquasi}
Let $(H, \mu _H, \Delta _H, \alpha _H, \sigma )$ be a coquasitriangular Hom-bialgebra 
satisfying the condition 
$\sigma =\sigma \circ (\alpha _H\ot \alpha _H)$. \\
(i) Let $(M, \alpha _M)$ be a left $H$-comodule with coaction $M\rightarrow H\ot M$, 
$m\mapsto m_{(-1)}\ot m_{(0)}$. Define the linear map $H\ot M\rightarrow M$,  
$h\ot m\mapsto h\cdot m:=\sigma (m_{(-1)}\ot \alpha _H(h))m_{(0)}$. 
Then $(M, \alpha _M)$ with these structures is a Yetter-Drinfeld module over $H$.  \\
(ii) Assume that $\alpha _H$ is bijective. Let $(N, \alpha _N)$ be another left $H$-comodule 
with coaction $N\rightarrow H\ot N$, $n\mapsto n_{(-1)}\ot n_{(0)}$, 
regarded as a Yetter-Drinfeld module 
as in (i), via the map $H\ot N\rightarrow N$,  
$h\ot n\mapsto h\cdot n:=\sigma (n_{(-1)}\ot \alpha _H(h))n_{(0)}$.
We regard 
$(M\ot N, \alpha _M\ot \alpha _N)$ as a left $H$-comodule via the standard coaction 
$M\ot N\rightarrow H\ot (M\ot N)$, 
$m\ot n\mapsto m_{(-1)}n_{(-1)}\ot (m_{(0)}\ot n_{(0)})$ 
and then we regard 
$(M\ot N, \alpha _M\ot \alpha _N)$ as a Yetter-Drinfeld module as in (i). Then this 
Yetter-Drinfeld module coincides with the 
Yetter-Drinfeld module $M\tot N$ defined in Theorem \ref{mainmirror}.
\end{proposition}
\begin{theorem}
Let $(H, \mu _H, \Delta _H, \alpha _H, \sigma)$ be a coquasitriangular Hom-bialgebra with 
$\alpha _H$ bijective and 
$\sigma =\sigma \circ (\alpha _H\ot \alpha _H)$. Denote by $^H{\mathcal M}$ the category whose 
objects are left $H$-comodules $(M, \alpha _M)$ with $\alpha _M$ bijective and morphisms are 
morphisms of left $H$-comodules. Then $^H{\mathcal M}$ is a quasi-braided pre-tensor 
subcategory of $_H^H{\mathcal YD}$, with tensor product defined as 
in Proposition \ref{tensprodmod} (ii), associativity constraints defined by the formula 
$\mfa _{M, N, P}((m\ot n)\ot p)=\alpha _M(m)\ot (n\ot \alpha _P^{-1}(p))$, for $M, N, P\in \;
^H{\mathcal M}$, and quasi-braiding $c_{M, N}:M\ot N\rightarrow N\ot M$, 
$c_{M, N}(m\ot n)=\sigma (n_{(-1)}\ot m_{(-1)})\alpha _N^{-1}(n_{(0)})\ot 
\alpha _M^{-1}(m_{(0)})$, for all 
$M, N\in \;^H{\mathcal M}$.
\end{theorem}

We recall the following result  (\cite{homquantum2}, Theorem 7.4):
\begin{proposition} Let 
$(H, \mu _H, \Delta _H, \alpha _H, \sigma )$ be a coquasitriangular Hom-bialgebra such that  
$\sigma =\sigma \circ (\alpha _H\ot \alpha _H)$. 
If $(M, \alpha _M)$, $(N, \alpha _N)$ are left $H$-comodules, we define the linear map 
\begin{eqnarray*}
&&B_{M, N}:M\ot N\rightarrow N\ot M, \;\;\;B_{M, N}(m\ot n)=\sigma (n_{(-1)}\ot 
m_{(-1)})n_{(0)}\ot m_{(0)}. 
\end{eqnarray*}
Then $(\alpha _N\ot \alpha _M)\circ B_{M, N}=B_{M, N}\circ 
(\alpha _M\ot \alpha _N)$ and, 
if $(P, \alpha _P)$ is another left $H$-comodule, the maps $B_{-, -}$ 
satisfy the HYBE
\begin{eqnarray*}
&&(\alpha _P\ot B_{M, N})\circ (B_{M, P}\ot \alpha _N)\circ (\alpha _M\ot B_{N, P})\nonumber \\
&&\;\;\;\;\;\;\;\;\;\;\;\;\;=(B_{N, P}\ot \alpha _M)\circ (\alpha _N\ot B_{M, P})
\circ (B_{M, N}\ot \alpha _P). 
\end{eqnarray*}
\end{proposition}

It turns out that the particular case of this result in which $\alpha _H$ is bijective is a particular 
case of Proposition \ref{BYD}, via Proposition \ref{connectcoquasi}.

\end{document}